\documentclass[preprint,12pt]{elsarticle}

\usepackage{amssymb}

\usepackage{amsthm}
\usepackage{mathptmx}
\usepackage{graphicx,amssymb,mathrsfs,amsmath,hyperref}
\usepackage{lineno}

\journal{Finite Fields Appl. 50 (2018), 272-278. https://doi.org/10.1016/j.ffa.2017.12.004}

\newtheorem{thm}{Theorem}[section]
\newtheorem{cor}[thm]{Corollary}
\newtheorem{lem}[thm]{Lemma}

\theoremstyle{definition}

\theoremstyle{remark}
\newtheorem{rem}[thm]{Remark}

\newproof{pf}{Proof}
\newproof{pot}{Proof of Theorem \ref{thm2}}

\begin{document}

\begin{frontmatter}



\title{Normal bases and irreducible polynomials\tnoteref{}}


\author{Hua Huang}

\author{Shanmeng Han}

\author{Wei Cao\corref{cor1}}
\ead{caowei@nbu.edu.cn}
\cortext[cor1]{Corresponding author.}

\address{Department of Mathematics, Ningbo University, Ningbo, Zhejiang 315211, P.R. China}

\begin{abstract}
Let $\mathbb{F}_q$ denote the finite field of $q$ elements and $\mathbb{F}_{q^n}$ the degree $n$ extension of $\mathbb{F}_q$. A normal basis of $\mathbb{F}_{q^n}$ over $\mathbb{F} _q$ is a basis of the form $\{\alpha,\alpha^q,\dots,\alpha^{q^{n-1}}\}$. An irreducible polynomial in $\mathbb{F} _q[x]$ is called an $N$-polynomial if its roots are linearly independent over $\mathbb{F} _q$. Let $p$ be the characteristic of $\mathbb{F} _q$. Pelis et al. showed that every monic irreducible polynomial with degree $n$ and nonzero trace is an $N$-polynomial provided that $n$ is either a power of $p$ or a prime different from $p$ and $q$ is a primitive root modulo $n$. Chang et al. proved that the converse is also true. By comparing the number of $N$-polynomials with that of irreducible polynomials with nonzero traces, we present an alternative treatment to this problem and show that all the results mentioned above can be easily deduced from our main theorem.
\end{abstract}


\end{frontmatter}



\section{Introduction}

Let $p$ be a prime number, $n\geq2$ be an integer. Let $\mathbb{F} _q$
denote the finite field of $q$ elements with characteristic $p$,
and $\mathbb{F} _{q^n}$ be its extension of degree $n$. A {\em normal
basis} of $\mathbb{F} _{q^n}$ over $\mathbb{F} _q$ is a basis of the form
$\{\alpha,\alpha^q,\dots,\alpha^{q^{n-1}}\}$, i.e., a basis
consisting of all the algebraic conjugates of a fixed element. We
say that $\alpha$ generates a normal basis, or $\alpha$ is a
normal element of $\mathbb{F} _{q^n}$ over $\mathbb{F} _q$. In either case we are
referring to the fact that the elements
$\alpha,\alpha^q,\dots,\alpha^{q^{n-1}}$ are linearly independent
over $\mathbb{F} _q$. In 1850, Eisenstein \cite{Eisenstein} first
conjectured the existence of normal bases for finite fields, and
its proof was given by Sch$\rm \ddot{o}$nemann \cite{Schonemann}
later in 1850 for the case $\mathbb{F} _p$ and then by Hensel
\cite{Hensel} in 1888 for arbitrary finite fields. Normal bases
over finite fields have proved very useful for fast arithmetic
computations with potential applications to coding theory and to
cryptography, see, e.g., \cite{Gao,Lidl,Menezes}.

An irreducible polynomial in $\mathbb{F} _q[x]$ is called an {\em
$N$-polynomial} if its roots are linearly independent over $\mathbb{F}
_q$. The minimal polynomial of any element in a normal basis
$\alpha,\alpha^q,\dots,\alpha^{q^{n-1}}$ is
$m(x)=\prod_{i=0}^{n-1}(x-\alpha^{q^{i}})\in \mathbb{F} _q[x]$, which is
irreducible over $\mathbb{F} _q$. The elements in a normal basis are
exactly the roots of an $N$-polynomial. Hence an $N$-polynomial is
just another way of describing a normal basis. In general, it is
not easy to check whether an irreducible polynomial is an
$N$-polynomial. However in certain cases, the thing may be very simple
according to Theorems 1.1 and 1.2 below.

\begin{thm}\label{thm1}
{\rm (Pelis \cite{Perlis})} Let $n=p^e$ with $e\geqslant 1$. Then
an irreducible polynomial $f(x)=x^n+a_1x^{n-1}+\dots+a_n\in \mathbb{F}
_q[x]$ is an $N$-polynomial if and only if $a_1\neq0$.
\end{thm}

\begin{thm}\label{thm2}
{\rm (Pei, Wang and Omura \cite{Pei})} Let $n$ be a prime
different from $p$ and $q$ be a primitive root modulo $n$. Then
an irreducible polynomial $f(x)=x^n+a_1x^{n-1}+\dots+a_n\in \mathbb{F}
_q[x]$ is an $N$-polynomial if and only if $a_1\neq0$.
\end{thm}

In 2001, Chang, Truong and Reed \cite{Chang} furthermore proved
that the conditions in Theorems 1.1 and 1.2 are also necessary.

\begin{thm}\label{thm3}
If every irreducible polynomial $f(x)=x^n+a_1x^{n-1}+\dots+a_n\in
\mathbb{F} _q[x]$ with $a_1\neq0$ is an $N$-polynomial, then $n$ is either
a power of $p$ or a prime different from $p$ and $q$ is a
primitive root modulo $n$.
\end{thm}

By comparing the number of $N$-polynomials with that of irreducible polynomials over $\mathbb{F} _q$,
we will present an alternative treatment to this problem. Throughout the rest of the paper, write $n=mp^e$ with $p\nmid
m$. It will be seen that Theorems \ref{thm1}, \ref{thm2} and \ref{thm3} are all the direct consequences of the following theorem.

\begin{thm}{\rm (Main Theorem)}\label{mainthm}
  The following inequality holds
\begin{equation}
  q^{n-m}\prod_{d|m}(q^{\tau(d)}-1)^{\phi(d)/\tau(d)}
  \leqslant\frac{q-1}{q}\sum_{d|m}\mu(d)q^{n/d},  \label{eq1.1}
\end{equation}
where $\tau(d)$ is the order of $q$ modulo $d$, $\phi(d)$ is the
Euler totient function, and $\mu(d)$ is the M$\rm{\ddot{o}}$bius
function. Furthermore, \textup{(\ref{eq1.1})} becomes an equality if and
only if $n=p^e$, or $n$ is a prime different from $p$ and $q$ is a
primitive root modulo $n$.
\end{thm}

\section{Two counting formulae and partial proof of Theorem \ref{mainthm}}
In this section, we will first explain two counting formulae appearing in Theorem 1.4 and then give a partial proof to it.
Let $v(n,q)$ denote the number of normal elements of $\mathbb{F}
_{q^n}$ over $\mathbb{F} _q$. Hensel \cite{Hensel} and Ore \cite{Ore}
obtained an expression of $v(n,q)$ by the factorization of
$x^n-1$. Akbik \cite{Akbik} and Gathen and Giesbrecht
\cite{Gathen} gave the explicit formula for $v(n,q)$ as follows,
which need not factorize $x^n-1$:
\begin{thm} The number of normal elements of $\mathbb{F}
_{q^n}$ over $\mathbb{F} _q$ is given by
\begin{equation*}
    v(n,q)=q^{n-m}\prod_{d|m}(q^{\tau(d)}-1)^{\phi(d)/\tau(d)}.
\end{equation*}
\end{thm}

Since every element in a normal basis generates the same basis, we get
\begin{cor}\label{lem1}
The number of normal bases of $\mathbb{F} _{q^n}$ over $\mathbb{F} _q$ is given by
\begin{equation*}
    \frac{v(n,q)}{n}=\frac{q^{n-m}}{n}\prod_{d|m}(q^{\tau(d)}-1)^{\phi(d)/\tau(d)}.
\end{equation*}
\end{cor}

The {\em trace of a degree $n$ polynomial $f(x)$ over $\mathbb{F} _q$} is
defined to be the coefficient of $x^{n-1}$. For a given $t\in \mathbb{F}_q^*$, let $I_q(n,t)$ denote
the number of monic irreducible polynomials of degree $n$ over $\mathbb{F}
_q$ with trace $t$. Relying on a generalized
M$\rm{\ddot{o}}$bius inversion formula, Ruskey, Miers and Sawada
\cite{Ruskey} showed that
\begin{thm}Let $t\in \mathbb{F}_q^*$, then
\begin{equation}
    I_q(n,t)=\frac{1}{qn}\sum_{d|m}\mu(d)q^{n/d}. \label{eq2.3}
\end{equation}
\end{thm}
Observing that $I_q(n,t)$ in (\ref{eq2.3}) is independent of $t$, we have
\begin{cor}\label{lem2}
The number of monic irreducible polynomials of degree $n$ over $\mathbb{F}
_q$ with nonzero traces is given by
\begin{equation*}
    \sum_{t\in \mathbb{F} _q^*}I_q(n,t)=\frac{q-1}{qn}\sum_{d|m}\mu(d)q^{n/d}.
\end{equation*}
\end{cor}

{\em Partial Proof of Theorem \ref{mainthm}} Inequality (\ref{eq1.1}) comes from the fact that an $N$-polynomials must be an irreducible polynomial with nonzero trace and Corollaries \ref{lem1} and \ref{lem2}. Let $\mathfrak{L}$ and $\mathfrak{R}$ denote respectively the left- and right- hand sides of (\ref{eq1.1}). For the ``if" part, it is trivial to verify. So we focus on the ``only if" part. Given an integer $t$, denote by $v_q(t)$ the $q$-adic valuation of $t$. Assume $n=mp^e$ with $e\geq 1$ and $\mathfrak{L}=\mathfrak{R}$. We claim $m=1$. Otherwise, let $m=p_1^{t_1}\dots p_k^{t_k}$ be the standard form of prime factorization with $p_i$ distinct primes and $t_i\geq1$. We see that $v_q(\mathfrak{L})=n-m=p_1^{t_1}\dots p_k^{t_k}(p^e-1)$ and $v_q(\mathfrak{R})=p_1^{t_1-1}\dots p_k^{t_k-1}p^e-1$. If there is some $t_i\geq2$, then $p_i|v_q(\mathfrak{L})$ but $p_i\nmid v_q(\mathfrak{R})$, a contradiction. So $t_1=\dots t_k=1$, which yields $v_q(\mathfrak{L})>v_q(\mathfrak{R})$, a contradiction too. Thus our claim is proved. Now assume $e=0$ and $\mathfrak{L}=\mathfrak{R}$ and again let $n=p_1^{t_1}\dots p_k^{t_k}$ be the standard form of prime factorization. Since $0=v_q(\mathfrak{L})=v_q(\mathfrak{R})=p_1^{t_1-1}\dots p_k^{t_k-1}-1$, we have $t_1=\dots t_k=1$, i.e., $n$ must be square free. It remains to show that if $k\geq2$ or $k=1$ but $q$ is not a primitive root modulo $n$ then $\mathfrak{L}<\mathfrak{R}$. This work is nontrivial and we will devote the next section to it.
\begin{rem}
  Just from the inequality $\mathfrak{L}\leq\mathfrak{R}$ and the ``if" part, one can easily deduce Theorems \ref{thm1} and \ref{thm2}.
\end{rem}

\section{$q$-polynomials and continuation of proof of Theorem \ref{mainthm}}
To finish the proof of Theorem \ref{mainthm}, we need some results about $q$-polynomials that will be briefly listed below; refer to \cite[Chapters 2-3]{Lidl} for details. A polynomial of the form $L(x)=\sum_{i=0}^nc_ix^{q^i}\in \mathbb{F} _q[x]$ is called a \emph{$q$-polynomial} (or a \emph{linearized polynomial}), and its \emph{conventional $q$-associate} is defined to be $l(x)=\sum_{i=0}^nc_ix^i$. Given two $q$-polynomials $L_1(x)$ and $L_2(x)$, we define \emph{symbolic multiplication} by $L_1(x)\bigotimes L_2(x)=L_1(L_2(x))$. Similarly, we can define \emph{symbolic division}, \emph{symbolic factorization}, \emph{symbolic irreducibility}, etc. for $q$-polynomials.
\begin{thm}{\rm (\cite[Lemma 3.59]{Lidl})}\label{ffirst}
  Let $L_1(x)$ and $L_2(x)$ be $q$-polynomials over $\mathbb{F}_q$ with conventional $q$-associates $l_1(x)$ and $l_2(x)$. Then $l(x)=l_1(x)l_2(x)$ and $L(x)=L_1(x)\bigotimes L_2(x)$ are $q$-associates of each other.
\end{thm}

The following criterion is an immediate consequence of the theorem above.
\begin{cor}\label{cor3}
  Let $L_1(x)$ and $L(x)$ be $q$-polynomials over $\mathbb{F}_q$ with conventional $q$-associates $l_1(x)$ and $l(x)$. Then $L_1(x)$ symbolically divides $L(x)$ if and only if $l_1(x)$ divides $l(x)$. In particular, $L(x)$ is symbolically irreducible over $\mathbb{F}_q$ if and only if $l_1(x)$ is irreducible over $\mathbb{F}_q$.
\end{cor}

Let $L(x)\in\mathbb{F}_q[x]$ be a nonzero $q$-polynomial with $q^n$ simple roots and let
\begin{equation}\label{sff}
 \mathfrak{F}: \quad L(x)=\underbrace{L_1(x)\bigotimes\cdots\bigotimes L_1(x)}_{e_1}\bigotimes\cdots\bigotimes \underbrace{L_r(x)\bigotimes\cdots\bigotimes L_r(x)}_{e_r}
\end{equation}
be the symbolic factorization of $L(x)$ with pairwise relatively prime (not necessarily symbolically irreducible) $q$-polynomials $L_i(x)$ over $\mathbb{F}_q$ and $\deg(L_i(x))=q^{n_i}$. For $1\leq i\leq r$, denote by $K_i(x)$ the polynomial obtained from the symbolic factorization $\mathfrak{F}$ in (\ref{sff}) of $L(x)$ by omitting the symbolic factor $L_i(x)$. Clearly, $\deg(K_i(x))=q^{n-n_i}$. Let $S_\mathfrak{F}(L)$ be the set of the roots of $L(x)$ that are not the roots of some $K_i(x)$. By the inclusion-exclusion principle, the formula for $|S_\mathfrak{F}(L)|$ is given by
\begin{align}
|S_\mathfrak{F}(L)|= & \ q^n-\sum_{i=1}^rq^{n-n_i}+\sum_{1\leq i<j\leq r}q^{n-n_i-n_j}-\dots+(-1)^rq^{n-n_1-\cdots-n_r}\nonumber\\
   = & \ q^n(1-q^{-n_1})\cdots(1-q^{-n_r}).\label{ffn}
\end{align}
The expression can also be interpreted in a different way. Let $l(x)$ be the conventional $q$-associate of $L(x)$. Then
\begin{equation}\label{rrr}
 \mathfrak{f}: \quad   l(x)=l_1(x)^{e_1}l_2(x)^{e_2}\cdots l_r(x)^{e_r}
\end{equation}
is the canonical factorization of $l(x)$ in $\mathbb{F}_q[x]$, where $l_i(x)$, the conventional $q$-associate of $L_i(x)$, are pairwise relatively prime (not necessarily irreducible) and $\deg(l_i(x))=n_i$. Let $S_\mathfrak{f}(l)$ denote the set of polynomials in $\mathbb{F}_q[x]$ that are of smaller degree than $l(x)$ as well as not divisible by any $l_i(x)$ for all $1\leq i\leq r$ in the factorization $\mathfrak{f}$ in (\ref{rrr}). The \emph{generalized Euler's $\phi$-function} for $l(x)$ under such factorization is defined to be $\Phi_\mathfrak{f}(l):=|S_\mathfrak{f}(l)|$.
\begin{rem}
The sets $S_\mathfrak{F}(L)$ and $S_\mathfrak{f}(l)$ depend on the concrete factorizations $\mathfrak{F}$ and $\mathfrak{f}$, respectively. If all the polynomials $L_i(x)$ in the symbolic factorization $\mathfrak{F}$ are symbolically irreducible, then all the polynomials $l_i(x)$ in the factorization $\mathfrak{f}$ are irreducible in the ordinary sense by Corollary \ref{cor3} and we simply write $\Phi(l):=\Phi_\mathfrak{f}(l)$ for this case which coincides with the original definition in \cite[p. 122]{Lidl}.
\end{rem}

\begin{thm}\label{thm11}
  \begin{enumerate}
    \item[{\rm(i)}] $|S_\mathfrak{F}(L)|=\Phi_\mathfrak{f}(l)$.
    \item[{\rm(ii)}] Suppose that $l_1(x)=g(x)h(x)$ where $g(x)$ and $h(x)$ are relatively prime with $\deg(g)\geq1$ and $\deg(h)\geq1$. Then for the new factorization
    \begin{equation*}
         \mathfrak{f}': \quad   l(x)=g(x)^{e_1}h(x)^{e_1}l_2(x)^{e_2}\cdots l_r(x)^{e_r}
    \end{equation*}
          we have $\Phi_\mathfrak{f'}(l)<\Phi_\mathfrak{f}(l)$. In particular, $\Phi(l)\leq\Phi_\mathfrak{f}(l)$.
  \end{enumerate}
\end{thm}
\begin{proof}
  (i) The proof  is similar to the proof of \cite[Lemma 3.69 (iii)]{Lidl}.

  (ii) Clearly, $S_\mathfrak{f'}(l)\subset S_\mathfrak{f}(l)$ and $g(x)\in S_\mathfrak{f}(l)$ but $g(x)\not\in S_\mathfrak{f'}(l)$. So $S_\mathfrak{f'}(l)\varsubsetneq S_\mathfrak{f}(l)$ and hence $\Phi_\mathfrak{f'}(l)<\Phi_\mathfrak{f}(l)$.
\end{proof}

Recall the assumption that $n=mp^e$ with $p\nmid m$ and it remains to consider the case that $e=0$ and $n$ is square free.
\begin{lem}{\rm (\cite[Theorems 2.45 and 2.47]{Lidl})}  \label{ffcc}
For $e=0$, i.e., $p\nmid n$, we have
\begin{equation*}
   x^{n}-1=\prod_{d|n}\varphi_d(x)=\prod_{d|n}\prod_{j=1}^{\phi(d)/\tau(d)}h_{d{_j}}(x),
\end{equation*}
where $\varphi_d(x)$ denotes the $d$th cyclotomic polynomial of degree $\phi(d)$, and has $\phi(d)/\tau(d)$ distinct monic irreducible factors $h_{d{_j}}(x)$, each of degree $\tau(d)$.
\end{lem}

{\em Continue Proof of Theorem \ref{mainthm}} Let $n=p_1\dots p_k$ where $p_i$ are distinct primes different from $p$, $k\geq2$ or $k=1$ but $q$ is not a primitive root modulo $n$. We need to show $\mathfrak{L}<\mathfrak{R}$ where $\mathfrak{L}$ and $\mathfrak{R}$ denote the left- and right- hand sides of (\ref{eq1.1}), respectively. By Theorem  \ref{thm11} and Lemma \ref{ffcc}, we have
\begin{equation}\label{ww}
  \Phi(x^{n}-1)= \prod_{d|n}(q^{\tau(d)}-1)^{\phi(d)/\tau(d)}\leq\prod_{d|n}(q^{\phi(d)}-1).
\end{equation}
Furthermore, the second inequality in (\ref{ww}) becomes a strict inequality if $\phi(d)>\tau(d)$ for some divisor $d$, especially, $k=1$ but $q$ is not a primitive root modulo $n$ . Therefore it suffices to prove the strict inequality
\begin{equation}\label{rr11}
  \prod_{d|n}(q^{\phi(d)}-1)<\frac{q-1}{q}\sum_{d|n}\mu(d)q^{n/d}
\end{equation}
holds for $k\geq2$. Let $\Psi_d(x)$ denote the linearized $q$-associate of $\varphi_d(x)$. By Theorem \ref{ffirst}, the corresponding symbolic factorization of $x^{n}-1=\prod_{d|n}\varphi_d(x)$ is
\begin{equation}\label{ffdd}
  \mathfrak{F}:\quad x^{q^n}-x=\bigotimes_{d|n} \Psi_d(x).
\end{equation}
Let $\alpha\in S_\mathfrak{F}(x^{q^n}-x)$. We claim three assertions as below:
\begin{enumerate}
  \item[(i)]The trace of $\alpha$ is nonzero, i.e.,  $\alpha^{q^{n-1}}+\cdots+\alpha^q+\alpha\neq0$.
  \item[(ii)] The degree of $\alpha$ is $n$, i.e., $n$ is the least positive integer $t$ such that $\alpha^{q^t}=\alpha $.
  \item[(iii)] There exists at least one element $\beta\not\in S_\mathfrak{F}(x^{q^n}-x)$ that satisfies both the above conditions (i) and (ii).
\end{enumerate}
Assume that $\alpha^{q^{n-1}}+\cdots+\alpha^q+\alpha=0$. It follows that $\alpha$ is a root of $\bigotimes_{1<d|n} \Psi_d(x)= x^{q^{n-1}}+\cdots+x^q+x$ and this contradicts to the definition of $S_\mathfrak{F}(x^{q^n}-x)$. So assertion (i) is proved. Now assume that  $l<n$ is the least positive integer $t$ such that $\alpha^{q^t}=\alpha$. Then we have $l|\frac{n}{p_i}$ for some $1\leq i\leq k$ and hence $\alpha$ is a root of $\bigotimes_{d|\frac{n}{p_i}} \Psi_d(x)=x^{q^{\frac{n}{p_i}}}-x$, again contradicting to the definition of $S_\mathfrak{F}(x^{q^n}-x)$. So assertion (ii) is true. Assertions (i) and (ii) means that $\alpha$ and its conjugates over $\mathbb{F}_q$ form the roots of an irreducible polynomial with nonzero trace over $\mathbb{F}_q$. Thus by Corollary \ref{lem2} and Theorem \ref{thm11}, we have
\begin{equation}\label{666}
  \# S_\mathfrak{F}(x^{q^n}-x)=\prod_{d|n}(q^{\phi(d)}-1)\leq \frac{q-1}{q}\sum_{d|n}\mu(d)q^{n/d}.
\end{equation}

To prove assertion (iii), we observe that the element $\alpha \in S_\mathfrak{F}(x^{q^n}-x)$ also satisfies other conditions besides (i) and (ii), e.g., $\Psi_1(\alpha)\bigotimes\Psi_n(\alpha)\neq0$ for $k\geq2$. Set $A=\{\beta\in \mathbb{F}_{q^n}: \Psi_1(\beta)\bigotimes\Psi_n(\beta)=0\}$, $A_0=\{\beta\in A: \bigotimes_{1<d|n} \Psi_d(\beta)=0\}$, and $A_i=\{\beta\in A: \bigotimes_{d|\frac{n}{p_i}} \Psi_d(\beta)=0\}$ for $i=1,\dots,k$. Then by the inclusion-exclusion principle, we obtain
\begin{equation*}
\big|A\setminus \bigcup_{i=0}^k A_i\big|= q^{1+\varphi(n)}-q^{\varphi(n)}+\sum_{i=1}^k{k \choose i}(-1)^i q=q^{1+\varphi(n)}-q^{\varphi(n)}-q\geq1.
\end{equation*}
This proves assertion (iii) which says that (\ref{666}) can not be an equality. Thus the strict inequality (\ref{rr11}) holds. The proof is finished.

\section*{Acknowledgements}
 This work was jointly supported by the National Natural Science Foundation of China (11371208), Zhejiang Provincial Natural Science Foundation of China (LY17A010008) and Ningbo Natural Science Foundation (2017A610134), and sponsored by the K. C. Wong Magna Fund in Ningbo University.


\begin{thebibliography}{20}
\bibitem{Akbik}
S.~Akbik, Normal generators of finite fields, J.
Number Theory, 41 (1992), 146-149.
\bibitem{Chang}
Y.~Chang, T.~K. Truong, I.~S. Reed, Normal bases
over $GF(q)$, J. Algebra, 241 (2001), 89-01.

\bibitem{Eisenstein}
G.~Eisenstein, Lehr$\rm{\ddot{a}}$tze, J. Reine Angew.
Math. 39 (1850), 180-182.

\bibitem{Gao}
S.~Gao. Normal Bases over Finite Fields. PhD Thesis, Univ. Waterloo, Ontario, 1993.

\bibitem{Gathen}
J.~Gathen, M.~Giesbrecht, Constructing normal bases in
finite fields, J. Symbolic Computation, 10 (1990), 547-570.

\bibitem{Hensel}
K.~Hensel, $\rm{\ddot{U}}$ber die Darstellung der
Zahlen eines Gattungsbereiches f$\rm{\ddot{u}}$r einen beliebigen
Primdivisor, J. Reine Angew. Math. 103 (1888), 230-237.

\bibitem{Lidl} R.~Lidl, H.~Niederreiter, Finite Fields, Encyclopedia of Mathematics and its Applications, Vol. 20, Addison-Wesley, Reading, MA, 1983.

\bibitem{Menezes}
A. J.~Menezes, I. F.~Blake, X.~Gao, R. C.~Mullin, S. A.~Vanstone, T.~Yaghoobian. Applications of Finite Fields. Kluwer Academic Publishers, Dordrecht, 1993.

\bibitem{Ore}
O.~Ore, Contributions to the theory of finite fields,
Trans. Amer. Math. Soc. 36 (1934), 243-274.

\bibitem{Pei}
D.~Pei, C.~Wang, J.~Omura, Normal bases of finite
field $GF(2^m)$, IEEE Trans. Inform. Theory 32 (1986),
285-287.

\bibitem{Perlis}
S.~Perlis, Normal bases of cyclic fields of prime power
degree, Duke. Math. J. 9 (1942), 507-517.

\bibitem{Ruskey}
F.~Ruskey, C.~R. Miers, J.~Sawada, The number of
irreducible polynomials and Lyndon words with given trace, SIAM
J. Discrete Math. 14 (2001), 240-245.

\bibitem{Schonemann}
T.~Sch$\rm\ddot{o}$nemann, $\rm{\ddot{U}}$ber einige
von Herrn Dr. Eisenstein aufgestellte Lehr$\rm{\ddot{a}}$tze,
Irredutible congruenzen betreffend, J. Reine Angew. Math. 40
(1850), 185-187.

\end{thebibliography}
\end{document}